\documentclass[12pt,reqno]{amsart}
\usepackage[usenames,dvipsnames]{xcolor}
\usepackage{amsmath, amsfonts, epsfig, amssymb, graphics}
\usepackage[colorlinks=true,citecolor=cyan]{hyperref}
\usepackage[mathscr]{euscript}
\usepackage{shuffle}

\numberwithin{equation}{section}

\def\area{\mathrm{area}}

\def\S{{\mathbb S}}

\def\mbf#1{{\mathbf #1}}
\def\part#1#2{\bleu{\rouge{\{} \{#1\}\rouge{,}\{#2\}\rouge{\}}}}

\def\auteur#1{{\sc #1}}
\def\titreref#1{{\em #1}}

\newcommand{\Park}{\mathcal{P}}

\def\Dyck#1{\mathscr{D}_{#1}}

\newdimen\carrelength
\def\jcarre{\jaune{\linethickness{\carrelength}\line(1,0){.85}}}

\def\define#1{\bleu{\bf #1}}

\def\bleu{\textcolor{blue}}
\def\rouge{\textcolor{red}}
\def\jaune{\textcolor{yellow}}

\def\vert#1{{\color{LimeGreen} #1}}

\newtheorem{theorem}{\bleu{Theorem}}
\newtheorem*{thm}{\bleu{Theorem}}
\newtheorem{proposition}[theorem]{\bleu{Proposition}}

\def\jcarre{\put(0,0){\jaune{\linethickness{\carrelength}\line(1,0){1}}}}

\def\sud#1#2#3{\put(#1,#2){\rouge{\line(0,-1){1}}\put(-.45,-.7){\rouge{$\scriptstyle#3$}}}}
\def\est#1#2#3{\put(#1,#2){\bleu{\line(1,0){1}}\put(-.6,.15){\bleu{$\scriptstyle#3$}}}}

\begin{document} 

\title[Bounded Height]{\bleu{\large Bounded Height Interlaced Pairs of Parking Functions}}
  \author[F.~Bergeron]{Fran\c{c}ois Bergeron}
\address{D\'epartement de Math\'ematiques, UQAM,  C.P. 8888, Succ. Centre-Ville, 
 Montr\'eal,  H3C 3P8, Canada.}\date{January 2015. This work was supported by NSERC-Canada.}
 \email{bergeron.francois@uqam.ca}

\begin{abstract}
We enumerate interlaced pairs of parking functions whose underlying Dyck path has a bounded height. We obtain an explicit formula for this enumeration in the form of a quotient of analogs of Chebicheff polynomials having coefficients in the ring of symmetric functions.
\end{abstract}

\maketitle
 \parskip=0pt

{ \setcounter{tocdepth}{1}\parskip=0pt\footnotesize \tableofcontents}
\parskip=8pt  


\section*{Introduction}

The enumeration of bounded height Dyck paths is a long standing subject (at least tracing back to Kr\'ew\'eras~\cite{koshy,kreweras,deutsch}), with interesting ties to the Average Complexity Analysis of Algorithms (see~\cite{flajolet}) and Statistical Mechanics (see~\cite{rechnitzer}). One of its striking features is the fact that the corresponding generating functions, which are easily seen to be rational, are quotients of consecutive (renormalized) Chebicheff polynomials. This is a fact that seems to have been often rediscovered under different guise. 

On the other hand, there has recently been a lot of interest in parking functions. This is certainly due in part to their central role in the combinatorial study of Diagonal coinvariant $\S_n$-modules, introduced by Garsia and Haiman (see~\cite{garsia}). In the flurry of recent work on the subject (see~\cite{livre,haglund}), the notion (explicitly described in the sequel) of ``interlaced pairs of parking functions'' has recently emerged (see~\cite{aval}). In fact, families of such interlaced pairs are constructed for any given (underlying) rectangular Dyck path.

We here consider the enumeration of interlaced pairs of parking functions for which the underlying (classical) Dyck path has a bounded height. We obtain an explicit formula for this enumeration. This formula takes the form of the quotient of analogs of Chebicheff polynomials, having coefficients in the ring of symmetric functions. This allows an explicit description of the character of the corresponding $\S_n\times \S_n$-module, where the group $\S_n\times \S_n$ acts by permutation on each components of the interlaced pair. As a corollary, we obtain a formula for the enumeration of parking functions of bounded height.

\section{Parking functions}\label{sec_park}
 Recall that, for a given $n$, a \define{parking function} is sequence $\pi=\pi_1\pi_2\cdots \pi_n$ of integers $0\leq \pi_k\leq n$,
whose decreasing reordering $a_1a_2\cdots a_n$ is such that
    $$a_k\leq n-k.$$
In other words, this last sequence is a partition (in french notation) contained in the staircase shape $(n-2) (n-1)\cdots 1$. Such partitions are also known as \define{Dyck paths} (See Figure~\ref{fig1} for an example). 
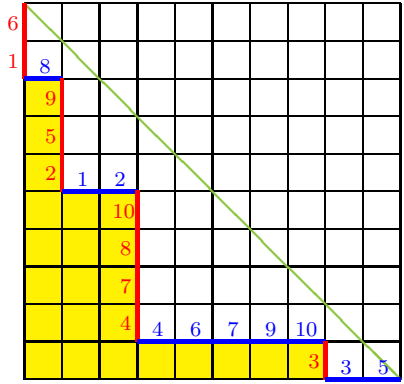
\begin{figure}[ht]
\begin{center}
\begin{picture}(10,10)(0,0)
\put(0,.5){
  \multiput(0,5)(0,1){3}{\jcarre}
  \multiput(0,0)(0,1){5}{\multiput(0,0)(1,0){3}{\jcarre}}
  \multiput(3,0)(1,0){5}{\jcarre}}
\multiput(0,0)(0,1){11}{\line(1,0){10}}
\multiput(0,0)(1,0){11}{\line(0,1){10}}
\thicklines
 \put(0,10){\vert{\line(1,-1){10}}}
  \linethickness{.5mm}
\sud{0}{10}{6}
\sud{0}{9}{1} \est{0}{8}{8}
\sud{1}{8}{9}
\sud{1}{7}{5}
 \sud{1}{6}{2}\est{1}{5}{1}\est{2}{5}{2}
\sud{3}{5}{\!\!10}
\sud{3}{4}{8}
\sud{3}{3}{7}
\sud{3}{2}{4}\est{3}{1}{4}\est{4}{1}{6}\est{5}{1}{7}\est{6}{1}{9}\est{7}{1}{\!\!10}
\sud{8}{1}{3}\est{8}{0}{3}\est{9}{0}{5}
\end{picture}\end{center}
\caption{An interlaced pair of parking functions on the Dyck path $8333311100$.}
\label{fig1}
\end{figure}

Conversely, to each given Dyck path $\alpha=a_{1} a_{2}\cdots a_{n}$, we may associate the set $\Park(\alpha)$ of \define{$\alpha$-parking function}:
\begin{displaymath}
   \bleu{\Park(\alpha):=\{a_{\sigma(1)} a_{\sigma(2)}\cdots a_{\sigma(n)} \ |\ \sigma\in\S_n\}}.
    \end{displaymath}
For $\pi\in\Park(\alpha)$, we say that $\alpha$ is the \define{shape} of $\pi$. Observe that $\alpha$-parking functions may be identified with standard Young tableaux\footnote{Naturally using french notation.} of skew shape $(\alpha+1^n)/\alpha$, where $\alpha+1^n$ is the partition having parts $a_k+1$. Simply put, $k$ sits in row $j$ of $(\alpha+1^n)/\alpha$, if and only if $\pi_k:=a_j$. In this way, an $\alpha$-parking function may be considered as a labelling of the vertical steps of the Dyck path $\alpha$.
      
By definition the symmetric group acts transitively  on $\Park(\alpha)$, by permutation of the $\pi_k$. The stabilizer of  $\alpha$ (considered as a  special case of parking-function) is clearly the Young subgroup
$\S_{r_0}\times \S_{r_1}\times \cdots\times \S_{r_k}$,
where $r_i$ denotes the number of occurrences of $i$ in $\alpha$.  It follows that the number of $\alpha$-parking function is given by the multinomial coefficient
   \begin{equation}\label{formdim}
       \bleu{\#\Park(\alpha)=  \binom{n}{r_0,r_1\ldots,r_k}}.
   \end{equation}
It also follows, by well known principles, that the Frobenius characteristic\footnote{This is the $\S_n$-character induced from the trivial character on the above Young subgroup.} of the resulting action of $\S_n$ on $\Park(\alpha)$ is given by the symmetric function product
   \begin{equation}\label{formchar}
        \bleu{\alpha(\mbf{x}):=\prod_{i=1}^k h_{r_i}(\mbf{x})}.
    \end{equation}
Hence, formula~\ref{formdim} is to be understood as giving the dimension of the corresponding $\S_n$-module. As is often done, one drops mention of the variables in symmetric functions. Thus we will now on write $h_k$ and $e_k$ respectively, for the \emph{complete homogeneous} and \emph{elementary} symmetric functions in the variables $\mbf{x}=x_1,x_2,x_3,\ldots$

\section{Bounded height parking functions} 
We are interested in studying the set of parking functions whose shape is some given Dyck path of height bounded by an integer $\eta$. Recall that the \define{height} of a path $\alpha=a_1a_2\cdots a_n$, denoted by $\eta(\alpha)$, is the maximum value of $n-k-a_k$, for $k$ running from $1$ to $n$. Our aim is to find explicit expressions for
  $$\bleu{\Park_n^{(\eta)}(\mbf{x};q) :=\sum_{\eta(\alpha)\leq \eta} q^{\area(\alpha)} \alpha(\mbf{x})},$$
 which we interpret as the Frobenius characteristic of the \define{area graded $\S_n$-module of parking functions of $\eta$-bounded-height}. Recall  that the \define{area} of a Dyck path is defined as
    $$\bleu{\area(\alpha):=\sum_{k=1}^n (n-k-\alpha_k)}.$$
In fact, we derive a formula for the generating series
       \begin{equation}
          \bleu{ \mathcal{P}^{(\eta)}(z;q):=\sum_{n=0}^\infty
                    \mathcal{P}_n^{(\eta)}(\mathbf{x};q)z^n}.
         \end{equation}
Indeed, just as is the case for the classical enumeration of bounded Dyck paths (see~\cite{flajolet}), the series $\mathcal{P}^{(\eta)}(z;q)$ takes the form of the quotient of two consecutive ``Chebicheff-like'' polynomials in $z$. The coefficients of these polynomials are themselves polynomials in the formal parameter $q$ and the elementary symmetric function $e_k$ rather that the $h_k$. This makes the formulas look simpler, but the implied expression for the $ \mathcal{P}_n^{(\eta)}$ should naturally be expanded back in terms of the $h_k$.
\begin{proposition}
 For all fixed $\eta$, as bound for the height, we have
     \begin{equation}
          \bleu{ \mathcal{P}^{(\eta)}(z;q)=\frac{\mathcal{T}_{\eta-1}(q\,z;q)}{\mathcal{T}_{\eta}(z;q)}},
         \end{equation}
where the polynomials $\mathcal{T}_{\eta}(z;q)$ are characterized by the recurrence
   $$\bleu{\mathcal{T}_{\eta+1}(z;q) = \sum_{i=0}^{\eta+2} (-1)^i q^{\binom{i}{2}} e_i \,z^i\,
                       \mathcal{T}_{\eta-i}(q^{i+1}\,z;q)},$$
with initial conditions $\mathcal{T}_{k}(z;q)=1$, for $k$ equal to $-2$ or $-1$. 
\end{proposition}
For instance, we have
\begin{eqnarray*}
    \bleu{\mathcal{T}_{{0}}(z;q)}&=&\bleu{1-e_1\,z},\\
     \bleu{\mathcal{T}_1(z;q)}&=&\bleu{1- \left( q+1 \right) e_1\,z+q\,e_{{2}}\,{z}^{2}},\\
      \bleu{\mathcal{T}_{{2}}(z;q)}&=&\bleu{1- \left( {q}^{2}+q+1 \right) e_1\,z+ \left( {q}^{2}{e_1
}^{2}+q \left( {q}^{2}+1 \right) e_{{2}} \right) {z}^{2}-{q}^{3}e_{{3}}\,
{z^{3}}}.
\end{eqnarray*}
We will see that this proposition is a corollary of Theorem~\ref{theo:bounded_xy} of the next section, where we consider the refined notion of ``interlaced pairs'' of parking function.



\section{Interlaced pairs of parking functions}
On top of labelling vertical steps of Dyck paths, we now add labels on their horizontal steps. See~\cite{aval} where this notion is introduced in a more general framework. To this end, we consider the \define{conjugate} $\alpha'$ of a Dyck path $\alpha$. This is simply the path associated to the conjugate partition in the usual sense. An \define{interlaced pair} of parking functions, on a given Dyck path $\alpha$, is simply a pair $(\rouge{\pi},\bleu{\pi'})\in\Park(\alpha)\times \Park(\alpha')$. Thus, it corresponds to independently labelling vertical steps and horizontal steps, as is illustrated in Figure~\ref{fig1}. We clearly have an action of $\rouge{\S_n}\times \bleu{\S_n}$ on the resulting set, and the corresponding character may be encoded as the product $\rouge{\alpha(\mbf{x})}\bleu{\alpha'(\mbf{y})}$. 

Any $\eta$-height-bounded path $\alpha$ is uniquely decomposed as a sequence of \define{hooks} $(\rouge{r_i},\bleu{s_i})$. These correspond to a maximal sequence of consecutive vertical steps (\rouge{$r_i$} of these), followed by maximal sequence of consecutive horizontal steps (\bleu{$s_i$} of those), so that
    $$\rouge{\alpha(\mbf{x})}\bleu{\alpha'(\mbf{y})}= \rouge{h_{r_1}(\mbf{x})}\bleu{h_{s_1}(\mbf{y})}\, \rouge{h_{r_2}(\mbf{x})}\bleu{h_{s_2}(\mbf{y})}\ \cdots\ 
              \rouge{h_{r_k}(\mbf{x})}\bleu{h_{s_k}(\mbf{y})}.$$
Reading them from top-left to bottom-right, the hooks of the path $\alpha$ in Figure~\ref{fig1} are
   $$(\rouge{2},\bleu{1}),\ (\rouge{3},\bleu{2}),\ (\rouge{4},\bleu{5}),\ (\rouge{1},\bleu{2}),$$
hence
   $$\rouge{h_\alpha(\mbf{x})}\bleu{\alpha'(\mbf{y})}=\rouge{h_{2}(\mbf{x})}\bleu{h_{1}(\mbf{y})}\, \rouge{h_{3}(\mbf{x})}\bleu{h_{2}(\mbf{y})}\, \rouge{h_{4}(\mbf{x})}\bleu{h_{5}(\mbf{y})}\, \rouge{h_{1}(\mbf{x})}\bleu{h_{2}(\mbf{y})}.$$
The \define{corners} of the path $\alpha$ are the points lying at both ends of such hooks, hence these include both extremities of the path.

Just as before, we want explicit expressions for the  area-graded bi-Frobenius of the $(\S_n\times \S_n)$-module of interlaced pairs parking functions of $\eta$-bounded-height, which is calculated/defined as
   \begin{equation}
        \bleu{\mathbb{P}_n^{(\eta)}(\mathbf{x},\mathbf{y};q):=\sum_{\eta(\alpha)\leq \eta} q^{\area(\alpha)} \alpha(\mathbf{x})\,
           \alpha'(\mathbf{y})},
  \end{equation}
where $\mathbf{y}=y_1,y_2,\ldots$ stands for another denumerable alphabet of variables, which is used for the symmetric functions that encode the second $\S_n$-action.  We then consider the generating series
       \begin{equation}
          \bleu{ \mathbb{P}^{(\eta)}(z;q):=\sum_{n=0}^\infty
                    \mathbb{P}_n^{(\eta)}(\mathbf{x},\mathbf{y};q)z^n}.
         \end{equation}
for which we have the formula of Theorem~\ref{theo:bounded_xy} below.
To state this theorem, we need to introduce the following ``transfer'' matrix, whose rows and columns are indexed by the possible heights of corners of the relevant paths. Thus, we consider the $(\eta+1)\times (\eta+1)$ matrix $A_{\eta}(\mbf{x},\mbf{y};q)$ having as entries 
	$$a_{i,j}^{(\eta)}(\mbf{x},\mbf{y}):=\begin{cases}
      \displaystyle\sum_{k=1}^{\eta+1-j} q^{\binom{k+j}{2}}h_k(\mbf{x})h_{k+j}(\mbf{y}),& \text{if}\ i=0, \\[15pt]
      \displaystyle\sum_{k=1}^{\eta+1-i} q^{\binom{k+1}{2}+(i-1)k}h_k(\mbf{x})h_{k+i}(\mbf{y}),& \text{if}\ j=0, \\[15pt]
       \displaystyle a_{i-1,j-1}^{(\eta-1)}(q\,\mbf{x},\mbf{y}),& \text{otherwise}.
\end{cases}$$
for $0\leq i,j\leq r$. Each $a_{i,j}^{(\eta)}=a_{i,j}^{(\eta)}(\mbf{x},\mbf{y})$ describes the possible ways in which one may go via a hook from a corner at height $i$ to a corner at height $j$, while respecting the height bound. The accompanying power of $q$ corresponds to the contribution to the area of that portion of the path. For example, we have
\def\r#1{\rouge{#1}}
\def\b#1{\bleu{#1}}
$$A_{2}(\mbf{x},\mbf{y};q)=\left[ \begin {array}{rrr} 
\scriptstyle 
\r{h_1(\mbf{x})} \b{h_1(\mbf{y})} +\r{h_2(\mbf{x})} \b{h_2(\mbf{y})}  q+ \r{h_3(\mbf{x})} \b{h_3(\mbf{y})}  {q}^{3}
 &\scriptstyle 
\r{h_2(\mbf{x})}  \b{h_1(\mbf{y})}  q+ \r{h_3(\mbf{x})} \b{h_2(\mbf{y})}  {q}^{3}
 &\scriptstyle 
\r{h_3(\mbf{x})} \b{h_1(\mbf{y})}  {q}^{3}
 \\ \noalign{\medskip}\scriptstyle 
 \r{h_1(\mbf{x})} \b{h_2(\mbf{y})} q + \r{h_2(\mbf{x})} \b{h_3(\mbf{y})} {q}^{3}
 &\scriptstyle 
\r{h_1(\mbf{x})}  \b{h_1(\mbf{y})} q + \r{h_2(\mbf{x})}  \b{h_2(\mbf{y})} {q}^{3}
 &\scriptstyle 
 \r{h_2(\mbf{x})}  \b{h_1(\mbf{y})} {q}^{3}
 \\ \noalign{\medskip} \scriptstyle 
 \r{h_1(\mbf{x})} \b{h_3(\mbf{y})} {q}^{2}
 &\scriptstyle 
 \r{h_1(\mbf{x})} \b{h_2(\mbf{y})} {q}^{2}
 &\scriptstyle 
\r{h_1(\mbf{x})}  \b{h_1(\mbf{y})} {q}^{2}\end {array} \right]
$$

  \begin{thm}\label{theo:bounded_xy}
    For all fixed $\eta$, as a bound for the height, we have
       \begin{equation}\label{eq:theo_xy}
          \bleu{ \mathbb{P}^{(\eta)}(z;q)=\frac{\mathbb{T}_{\eta-1}(q\,z;q)}{\mathbb{T}_{\eta}(z;q)}},
         \end{equation}
where the polynomials $\mathbb{T}_{\eta}(z;q)$ are 
      $$\bleu{\mathbb{T}_{\eta}(z;q):=\det(\mathrm{Id}_{\eta+1}-A_{\eta}(z\,\mbf{x},\mbf{y};q))}.$$
  \end{thm}
 
\begin{proof}[\bf Proof]  We first observe that the $\mbf{x}$-degree of the symmetric function involved coincides with the height of paths. Thus, $\alpha(z\,\mbf{x})\alpha'(\mbf{y})=\alpha(\mbf{x})\alpha'(\mbf{y})\,z^n$ for $\alpha$ in $\Dyck{n}$.

Now, since $\eta$-bounded-height paths decompose uniquely as sequences of hooks, and the matrix $A_{\eta}(\mbf{x},\mbf{y};q)$ describes all possible ways one may go, using a hook, from a given height to another, while respecting the height bound. it follows that the $(0,0)$-indexed entry of the matrix
    $$\mathrm{Id}_{\eta+1}+A_{\eta}+A_{\eta}^2+\ldots = (\mathrm{Id}_{\eta+1}-A_{\eta})^{-1},$$
enumerates all $\eta$-bounded-height paths. The theorem follows from Cramer's rule for this $(0,0)$-indexed entry, since the recursive definition of the matrix $A_{\eta}(\mbf{x},\mbf{y})$ makes it clear that its  $(0,0)$-minor coincides with $A_{\eta-1}(q\mbf{x},\mbf{y})$.
\end{proof}

Low degree examples of the polynomials $\mathbb{T}_{\eta}(z;q)$ are as follows
\begin{eqnarray*}
\mathbb{T}_{\eta}(z;q)&=&1-h_1(\mbf{x})h_1(\mbf{y})z\\
\mathbb{T}_{\eta}(z;q)&=&1- \left( q+1 \right) h_1(\mbf{x})h_1(\mbf{y})z+q \left( {h_1(\mbf{x})}^{2}h_1(\mbf{y})^{2}-h_2(\mbf{x})h_{{2}}(\mbf{y}) \right) {z}^{2}\end{eqnarray*}

\section{Further considerations}
As studied in~\cite{aval}, the notion of interlaced pairs of parking functions may be considered for Dyck paths in a $(m\times n)$-rectangle, with an action of the group $\S_m\times \S_n$. With an adequately defined notion of height, our considerations of bounded height may be extended to this more general framework. However, a general explicit description of the associated rational functions seems to be more complicated.


\end{document}